\documentclass[12pt]{article}
\usepackage{graphicx}
\usepackage{amsmath,amssymb,amsthm,amsfonts}
\newtheorem{thm}{Theorem}[section]
\newtheorem{cor}[thm]{Corollary}
\newtheorem{lem}[thm]{Lemma}
\newtheorem{prop}[thm]{Proposition}

\newcommand{\R}{{\mathbb{R}}}
\newcommand{\Z}{{\mathbb{Z}}}

\newcommand{\1}{\partial}
\newcommand{\2}{\overline}
\newcommand{\3}{\varepsilon}
\newcommand{\4}{\widetilde}

\begin{document}
\title{Minimizer of an isoperimetric ratio on\\ 
a metric on $\R^2$ with finite total area} 
\author{Shu-Yu Hsu\\
Department of Mathematics\\
National Chung Cheng University\\
168 University Road, Min-Hsiung\\
Chia-Yi 621, Taiwan, R.O.C.}
\date{April 9, 2011}
\smallbreak \maketitle
\begin{abstract}
Let $g=(g_{ij})$ be a complete Riemmanian metric on $\R^2$ with finite total 
area and $I_g=\inf_{\gamma}I(\gamma)$ with 
$I(\gamma)=L(\gamma)(A_{in}(\gamma)^{-1}
+A_{out}(\gamma)^{-1})$ where $\gamma$ is any closed simple curve in $\R^2$,
$L(\gamma)$ is the length of $\gamma$, $A_{in}(\gamma)$ 
and $A_{out}(\gamma)$ are the areas of the regions inside and outside 
$\gamma$ respectively, with respect to the metric $g$. Under some mild 
growth conditions on $g$ we prove the existence of a minimizer for $I_g$. As 
a corollary we obtain a new proof for the existence of a minimizer 
for $I_{g(t)}$ for any $0<t<T$ when the metric 
$g(t)=g_{ij}(\cdot,t)=u\delta_{ij}$ is the maximal solution of the Ricci 
flow equation $\1 g_{ij}/\1 t=-2R_{ij}$ on $\R^2\times (0,T)$ \cite{DH} where 
$T>0$ is the extinction time of the solution.
\end{abstract}

\vskip 0.2truein


\vskip 0.2truein
\setcounter{equation}{0}
\setcounter{section}{-1}

\setcounter{equation}{0}
\setcounter{thm}{0}

Isoperimetric inequalities arises in many problems on analysis and geometry
such as the study of partial differential equations and
Sobolev inequality \cite{B}, \cite{SY}, \cite{T1}. Isoperimetric 
inequalities are also used by N.S.~Trudinger \cite{T2} in the 
study of sharp estimates for the Hessian equations and Hessian integrals.
In \cite{G}, \cite{H1}, M.~Gage and R.~Hamilton studied isoperimetric 
inequalities arising from the curve shortening flow. In \cite{DH}, 
\cite{DHS} and \cite{H2}, P.~Daskalopoulos, R.~Hamilton, N.~Sesum, studied
isoperimetric inequalities in Ricci flow and used it to study the behavior
of solutions of Ricci flow which is an important tool in the classification
of manifolds \cite{MT}, \cite{P1}, \cite{P2}, \cite{Z}. 

Let $g=(g_{ij})$ be a complete Riemmanian metric on $\R^2$ with finite total 
area $A=\int_{\R^2}\,dV_g$ satisfying
\begin{equation}
\lambda_1(|x|)\delta_{ij}\le g_{ij}(x)\le\lambda_2(|x|)\delta_{ij}\quad
\forall |x|\ge r_0
\end{equation} 
for some constant $r_0>1$ and positive monotone decreasing functions 
$\lambda_1(r)$, $\lambda_2(r)$, on $[r_0,\infty)$ that satisfy 
\begin{equation}
\int_r^{c_0r}\sqrt{\lambda_1(\rho)}\,d\rho\ge\pi r\sqrt{\lambda_2(r)}
\quad\forall r\ge r_0,
\end{equation}
\begin{equation}
r\sqrt{\lambda_1(c_0r)}\ge b_1\int_r^{\infty}\rho\lambda_2(\rho)
\,d\rho\quad\forall r\ge r_0,
\end{equation}
\begin{equation}
\int_r^{r^2}\sqrt{\lambda_1(\rho)}\,d\rho\ge b_2\quad\forall r\ge r_0,
\end{equation}
and
\begin{equation}
\lambda_1(c_0r)\ge\delta\lambda_2(r)\quad\forall r\ge r_0
\end{equation}
for some constants $c_0>1$, $b_1>0$, $b_2>0$, $\delta>0$, where $|x|$ 
is the distance of $x$ from the origin with respect to the Euclidean metric.
For any closed simple curve $\gamma$ in $\R^2$, let (cf. \cite{DH})
\begin{equation}\label{isoperimetric-ratio}
I(\gamma)
=L(\gamma)\left(\frac{1}{A_{in}(\gamma)}+\frac{1}{A_{out}(\gamma)}\right)
\end{equation}
where $L(\gamma)$ is the length of the curve $\gamma$, $A_{in}(\gamma)$ 
and $A_{out}(\gamma)$ are the areas of the regions inside and outside 
$\gamma$ respectively, with respect to the metric $g$. Let
\begin{equation}\label{isoperimetric-ratio-inf}
I=I_g=\inf_{\tiny\begin{array}{c}\gamma\end{array}}I(\gamma)
\end{equation}
where the infimum is over all closed simple curves $\gamma$ in $\R^2$. 
In this paper we will prove that there exists a constant $b_0>0$ such 
that if the isoperimetric ratio $I_g<b_0$, then there exists a 
closed simple curve $\gamma$ satisfying $I_g=I(\gamma)$. As a 
corollary we obtain a new proof for the existence of a 
minimizer for the isoperimetric ratio $I_{g(t)}$ for any 
$0<t<T$ when the metric $g(t)=g_{ij}(\cdot,t)=u\delta_{ij}$ is the 
maximal solution of the Ricci flow \cite{DH} 
\begin{equation*}
\frac{\1 }{\1 t}g_{ij}=-2R_{ij}\quad\mbox{ on }\R^2\times (0,T)
\end{equation*} 
where $T>0$ is the extinction time of the solution and $u$ is a solution
of 
\begin{equation}\label{ricci-eqn}
u_t=\Delta\log u\quad\mbox{ on }\R^2\times (0,T).
\end{equation} 
We will use an adaptation of the technique of \cite{H1} and \cite{H2} to 
prove the result. In \cite{H1}, \cite{H2}, since the domain under 
consideration is either the sphere $S^2$ (\cite{H2}) or bounded domain in 
$\R^2$ (\cite{H1}), the minimizing sequences for the infimum of the 
isoperimetric ratios considered in \cite{H1}, \cite{H2}, stay in a compact 
set. On the other hand since the isoperimetric ratio 
\eqref{isoperimetric-ratio} is for any curve $\gamma$ in $\R^2$, the 
minimizing sequence of curves for the infimum of the isoperimetric 
ratio \eqref{isoperimetric-ratio-inf} may not stay in a compact subset of
$\R^2$ and may not have a limit at all. So we will need to show that there
exists a constant such that this is impossible when $I_g$ is less than this
constant. After this we will use the curve shortening flow technique of
\cite{H2} to modify the minimizing sequence of curves and show that they will
converge to a minimizer of \eqref{isoperimetric-ratio-inf}. 

For any $x_0\in\R^2$ and $r>0$ let $B_r(x_0)=\{x\in\R^2:|x-x_0|<r\}$ and 
$B_r=B_r(0)$. The main results of the paper are as follows.

\begin{thm}
Suppose $g$ satisfies (1) for some constant $r_0>1$ where $\lambda_1(r)$, 
$\lambda_2(r)$, are positive monotone decreasing functions on 
$[r_0,\infty)$ that satisfy (2), (3), (4) and (5) for some constants 
$c_0>1$, $b_1>0$, $b_2>0$ and $\delta>0$. Then there exists a constant 
$b_0>0$ depending on $b_1$, $b_2$ and $A$ such that the following holds.
If 
\begin{equation}\label{isoperimetric-ratio-condition}
I_g<b_0,
\end{equation} 
then there exists a closed simple curve 
$\gamma$ in $\R^2$ such that $I_g=I(\gamma)$. Hence $I_g>0$.
\end{thm}

\begin{prop}
Suppose $g=(g_{ij})$ satisfies 
\begin{equation*}\label{upper-lower-bd-metric}
\frac{C_1}{r^2(\log r)^2}\delta_{ij}\le g_{ij}
\le\frac{C_2}{r^2(\log r)^2}\delta_{ij}\quad\forall r\ge r_1
\end{equation*} 
for some constants $C_2\ge C_1>0$, $r_1>1$. Then there exist constants 
$c_0>1$, $\delta>0$, $b_1>0$, $b_2>0$, and $r_0\ge r_1$ such that (2), 
(3), (4) and (5) hold.
\end{prop}

\begin{cor}
Let $g_{ij}(x,t)=u(x,t)\delta_{ij}$ where $u$ is the maximal solution of
\eqref{ricci-eqn} with initial value $0\le u_0\in L^p(\R^2)\cap L^1(\R^2)$,
$u_0\not\equiv 0$, for some $p>1$ satisfying 
\begin{equation}\label{initial-condition}
u_0(x)\le\frac{C}{|x|^2 (\log|x|)^2}\quad\forall |x|>1
\end{equation}    
given by \cite{DP} and \cite{Hu} where $T=(1/4\pi)\int_{\R^2}u_0\,dx$. 
Then for any $0<t_1<T$ there exists a constant $b_0>0$ such that the following
holds. For any $t_1\le t<T$, if $I_{g(t)}<b_0$, then there exists a closed 
simple curve $\gamma$ that satisfies $I_{g(t)}=I(\gamma)$.
\end{cor}

\noindent{\it Proof of Proposition 2}: Let
$\lambda_i(r)=C_i(r\log r)^{-2}$, $i=1,2$, 
\begin{equation}
c_0=2e^{\pi\sqrt{C_2/C_1}},
\end{equation}
and $\delta=C_1/(2c_0^2C_2)$. We choose $r_2\ge r_1$ such that
\begin{equation}
\frac{\log r}{\log (c_0r)}\ge\frac{1}{\sqrt{2}}\quad\forall r\ge r_2.
\end{equation}
Then by (11) and (12), 
\begin{equation}
\frac{\lambda_1(c_0r)}{\lambda_2(r)}=\frac{C_1}{c_0^2C_2}
\left(\frac{\log r}{\log (c_0r)}\right)^2\ge\frac{C_1}{2c_0^2C_2}=\delta
\quad\forall r\ge r_2.
\end{equation}
We next note that 
\begin{equation}
\lim_{r\to\infty}\left((\log r)\log\left(\frac{\log (c_0r)}{\log r}\right)
\right) 
=\lim_{z\to 0}\frac{\log((\log c_0)z+1)}{z}=\log c_0.
\end{equation}
By (11) and (14) there exists $r_0\ge r_2$ such that 
\begin{equation}
(\log r)\log\left(\frac{\log (c_0r)}{\log r}\right)>\pi\sqrt{C_2/C_1}
\quad\forall r\ge r_0.
\end{equation}
By (13) and (15), we get (2) and (5). By (12) and a direct computation 
(3) and (4) holds with $b_1=\sqrt{C_1}/(\sqrt{2}c_0C_2)$, $b_2=\sqrt{C_1}
\log 2$, and the proposition follows.
\hfill$\square$\vspace{6pt}

\noindent{\it Proof of Corollary 3}: By \eqref{initial-condition} 
and the results of \cite{ERV} there exists a constant 
$C_2>0$ such that
\begin{equation}\label{u-upper-bd}
u(x,t)\le\frac{C_2}{|x|^2(\log |x|)^2}\quad\forall |x|>1,0<t<T
\end{equation}
and for any $t_0\in (0,T)$ there exists a constant $r_1>1$ such that
\begin{equation}\label{u-lower-bd}
u(x,t)\ge\frac{(3/2)t}{|x|^2(\log|x|)^2}\quad\forall |x|\ge r_1,0<t\le t_0.
\end{equation}
By \eqref{u-upper-bd}, \eqref{u-lower-bd}, Theorem 1 and Proposition 2, 
the corollary follows.
\hfill$\square$\vspace{6pt}

We will now assume that $g$ is a metric on $\R^2$ with finite total area
that satisfies (1), (2), (3), (4) and (5) for some constants $r_0>1$, $c_0>1$, 
$b_1>0$, $b_2>0$, $\delta>0$ where $\lambda_1(r)$, $\lambda_2(r)$, are positive 
monotone decreasing functions on $[r_0,\infty)$ for the rest of the paper. 
Let $b_0=\min (b_1,4b_2/A)$. Suppose \eqref{isoperimetric-ratio-condition} holds. 
Let $\{\gamma_k\}_{k=1}^{\infty}$ be a sequence of closed simple curves on 
$\R^2$ such that 
\begin{equation}\label{sequence-limit}
I(\gamma_k)\to I\quad\mbox{ as }k\to\infty\quad\mbox{ and }\quad
I(\gamma_k)<b_0\quad\forall k\in\Z^+.
\end{equation}
We will show that the sequence $\{\gamma_k\}_{k=1}^{\infty}$ is
contained in some compact set of $\R^2$.
Let $\Omega_k$ be the region inside $\gamma_k$ and $r_k=\min_{x\in\gamma_k}
|x|$. Let $L_e(\gamma_k)$ be the length of $\gamma_k$ and $|\Omega_k|$ 
be the area of $\Omega_k$ with respect to the Euclidean metric.
We choose $r_0'>r_0$ such that 
\begin{equation}\label{area-outside}
\mbox{Vol}_g(\R^2\setminus B_{r_0'})\le\frac{A}{4}\quad\forall k\in\Z^+.
\end{equation}
\begin{lem}
The sequence $r_k$ is uniformly bounded. 
\end{lem}
\begin{proof} Suppose the lemma is not true. Then there 
exists a subsequence of $r_k$ which we 
may assume without loss of generality to be the sequence itself such that 
\begin{equation}\label{rk-lower-bd}
r_k>r_0'\quad\forall k\in\Z^+
\end{equation}
and $r_k\to\infty$ as $k\to\infty$. Let $\4{\gamma}_k=\1 B_{r_k}$.
We choose a point $x_k\in\gamma_k\cap\1 B_{r_k}$ and let 
$\gamma_k:[0,2\pi]\to\R^2$ be a parametrization of the curve 
$\gamma_k$ such that $x_k=\gamma_k(0)=\gamma_k(2\pi)$. 
Since for any $k\in\Z^+$ either $0\in\Omega_k$ or $0\in\R^2\setminus
\Omega_k$ holds, thus either
\begin{equation}\label{0-in}
0\in\Omega_k\quad\mbox{ for infinitely many k}
\end{equation}
or 
\begin{equation}\label{0-out}
0\in\R^2\setminus\Omega_k\quad\mbox{ for infinitely many k}
\end{equation}
holds. We need the following result for the proof of the lemma.

\noindent $\underline{\text{\bf Claim 1}}$: There exists only finitely many
$k$ such that $\gamma_k\cap (\R^2\setminus\2{B}_{c_0r_k})\ne\emptyset$.
 
\noindent {\it Proof of Claim 1}: Suppose claim 1 is false. Then there 
exists infinitely many $k$ such that $\gamma_k\cap 
(\R^2\setminus\2{B}_{c_0r_k})\ne\emptyset$. Without loss of generality we 
may assume that 
\begin{equation}
\gamma_k\cap (\R^2\setminus\2{B}_{c_0r_k})\ne\emptyset\quad\forall k\in\Z^+.
\end{equation}  
By (23) there exists $\phi_0\in (0,2\pi)$ such that
$$
|\gamma_k(\phi_0)|>c_0r_k.
$$
Hence there exists $0<\phi_1<\phi_0<\phi_2<2\pi$ such that
$$
\gamma_k(\phi_1)=\gamma_k(\phi_2)=c_0r_k
$$
and
$$
r_k\le |\gamma_k(\phi)|\le c_0r_k\quad\forall\phi\in (0,\phi_1)
\cup (\phi_2,2\pi).
$$
Then by (1),
\begin{align}
L(\gamma_k)=&\int_0^{2\pi}(g_{ij}\overset{.}\gamma_k^i
\overset{.}\gamma_k^j)^{\frac{1}{2}}\,d\phi
\nonumber\\
\ge&\left(\int_0^{\phi_1}+\int_{\phi_2}^{2\pi}\right)
(g_{ij}\overset{.}\gamma_k^i\overset{.}\gamma_k^j)^{\frac{1}{2}}
\,d\phi\nonumber\\
\ge&\left(\int_0^{\phi_1}+\int_{\phi_2}^{2\pi}\right)
\sqrt{\lambda_1(r)}\sqrt{\biggl(\frac{dr}{d\phi}\biggr)^2
+r^2\biggl(\frac{d\theta}{d\phi}\biggr)^2}\,d\phi\nonumber\\
\ge&2\int_{r_k}^{c_0r_k}\sqrt{\lambda_1(r)}\,dr
\end{align} 
and
\begin{equation}
2\pi r_k\sqrt{\lambda_1(r_k)}
\le L(\4{\gamma}_k)=\int_0^{2\pi}
(g_{ij}\overset{.}{\4{\gamma}_k^i}\overset{.}{\4{\gamma}_k^j})^{\frac{1}{2}}
\,d\phi\le 2\pi r_k\sqrt{\lambda_2(r_k)}.
\end{equation}    
By (2), (24) and (25),
\begin{equation}
L(\4{\gamma}_k)\le L(\gamma_k).
\end{equation} 
Suppose \eqref{0-in} holds. Without loss of generality we may 
assume that $0\in\Omega_k$ for all $k\in\Z^+$. Then 
$B_{r_k}\subset\Omega_k$ for all $k\in\Z^+$. Hence by \eqref{area-outside}, 
\eqref{rk-lower-bd}, 
\begin{equation}
A_{out}(\gamma_k)\le\mbox{Vol}_g(\R^2\setminus B_{r_k})\le\frac{A}{4}
\quad\forall k\in\Z^+
\end{equation}  
and
\begin{equation} 
\frac{3A}{4}\le\mbox{Vol}_g(B_{r_k})\le A_{in}(\gamma_k)\le A
\quad\forall k\in\Z^+.
\end{equation} 
We will now show that the circle $\4{\gamma}_k=\1 B_{r_k}$ satisfies 
\begin{equation}
I(\4{\gamma}_k)\le I(\gamma_k). 
\end{equation}
Let $\3=A_{out}(\4{\gamma}_k)-A_{out}(\gamma_k)$. Then
$\3=A_{in}(\gamma_k)-A_{in}(\4{\gamma}_k)$. Since 
$\4{\gamma}_k\subset\2{\Omega}_k$ and
the region between $\gamma_k$ and $\4{\gamma}_k$ is contained in 
$\R^2\setminus B_{r_k}$, by (27),
\begin{equation}
0\le\3\le\frac{A}{4}.
\end{equation} 
Hence by (27) and (30),
\begin{align}
\frac{1}{A_{in}(\4{\gamma}_k)}+\frac{1}{A_{out}(\4{\gamma}_k)}
=&\frac{A}{A_{in}(\4{\gamma}_k)A_{out}(\4{\gamma}_k)}
=\frac{A}{(A_{in}(\gamma_k)-\3)(A_{out}(\gamma_k)+\3)}\nonumber\\
\le&\frac{A}{A_{in}(\gamma_k)A_{out}(\gamma_k)}
=\frac{1}{A_{in}(\gamma_k)}+\frac{1}{A_{out}(\gamma_k)}.
\end{align} 
By (26) and (31) we get (29). Now by (1),
\begin{equation}
A_{out}(\4{\gamma}_k)=\int_{\R^2\setminus B_{r_k}}\sqrt{\mbox{det}g_{ij}}\,dx
\le 2\pi\int_{r_k}^{\infty}\rho\lambda_2(\rho)\,d\rho.
\end{equation}
By (3), (25), (29) and (32),
\begin{equation}
I(\gamma_k)\ge\frac{L(\4{\gamma}_k)}{A_{out}(\4{\gamma}_k)}
+\frac{L(\4{\gamma}_k)}{A_{in}(\4{\gamma}_k)}\ge b_1.
\end{equation}
Letting $k\to\infty$ in (33),
\begin{equation}
I\ge b_1.
\end{equation}
This contradicts \eqref{isoperimetric-ratio-condition} and the definition of 
$b_0$. Hence \eqref{0-in} does not hold. 

Suppose \eqref{0-out} holds. Without loss of generality we may assume that 
$0\in\R^2\setminus\Omega_k$ for all $k\in\Z^+$. Then by \eqref{rk-lower-bd} 
$0\in\R^2\setminus
\2{\Omega}_k$ and $B_{r_k}\subset\R^2\setminus\2{\Omega}_k$ for any 
$k\in\Z^+$. By an argument similar to the proof
of (27) and (28) but with the role of $A_{in}(\gamma_k)$
and $A_{out}(\gamma_k)$ being interchanged in the proof we get
\begin{equation}
\left\{\begin{aligned}
&A_{in}(\gamma_k)\le\mbox{Vol}_g(\R^2\setminus B_{r_k})\le\frac{A}{4}
\quad\forall k\in\Z^+\\
&\frac{3A}{4}\le A_{out}(\gamma_k)\le A\qquad\qquad\quad\forall k\in\Z^+.
\end{aligned}\right.
\end{equation}
Similarly by interchanging the role of $A_{in}(\gamma_k)$
and $A_{out}(\gamma_k)$ and replacing $\3$ by $\3'
=A_{out}(\4{\gamma}_k)-A_{in}(\gamma_k)=A_{out}(\gamma_k)
-A_{in}(\4{\gamma}_k)$ in the proof of (29)--(33) above, we get that 
$0\le\3'\le A/4$ and (29), (33), still holds. Letting $k\to\infty$ in 
(33), we get (34). This again contradicts 
\eqref{isoperimetric-ratio-condition} and the definition of 
$b_0$. Thus \eqref{0-out} does not hold and claim 1 follows. 

We will now continue with the proof of the lemma. By claim 1 there exists 
$k_0\in\Z^+$ such that
\begin{align}
&\gamma_k\cap (\R^2\setminus\2{B}_{c_0r_k})=\emptyset\quad\forall k\ge k_0
\nonumber\\
\Rightarrow\quad&\gamma_k\subset\2{B}_{c_0r_k}\setminus B_{r_k}
\quad\forall k\ge k_0.
\end{align}
Note that either \eqref{0-in} or \eqref{0-out} holds. Suppose \eqref{0-in} holds.
Without loss of generality we may assume that $0\in\Omega_k$ for all 
$k\ge k_0$. Then $B_{r_k}\subset\Omega_k$ for all $k\ge k_0$. Hence
by (1) and (36),
\begin{align}
L(\gamma_k)=&\int_0^{2\pi}(g_{ij}\overset{.}\gamma_k^i\overset{.}\gamma_k^j)
^{\frac{1}{2}}\,d\phi\nonumber\\
\ge&\sqrt{\lambda_1(c_0r_k)}
\int_0^{2\pi}\left(\left(\frac{dr}{d\phi}\right)^2+r^2\left(
\frac{d\theta}{d\phi}\right)^2\right)^{\frac{1}{2}}\,d\phi\nonumber\\
\ge&2\pi r_k\sqrt{\lambda_1(c_0r_k)}\quad\forall k\ge k_0
\end{align}
and
\begin{equation}
A_{out}(\gamma_k)\le\int_{\R^2\setminus B_{r_k}}\sqrt{\mbox{det}g_{ij}}\,dx
\le 2\pi\int_{r_k}^{\infty}\rho\lambda_2(\rho)\,d\rho
\quad\forall k\ge k_0.
\end{equation}
By (3), (37) and (38),
\begin{equation}
I(\gamma_k)\ge\frac{L(\gamma_k)}{A_{out}(\gamma_k)}
\ge\frac{r_k\sqrt{\lambda_1(c_0r_k)}}
{\int_{r_k}^{\infty}\rho\lambda_2(\rho)\,d\rho}\ge b_1\quad\forall k\ge k_0.
\end{equation}
Letting $k\to\infty$ in (39), we get (34). Since (34) contradicts (9) and the 
definition of $b_0$, \eqref{0-in} does not hold. Hence \eqref{0-out} holds.
By \eqref{rk-lower-bd} and \eqref{0-out} we may assume without loss of 
generality that $0\in\R^2\setminus\2{\Omega}_k$ for all $k\ge k_0$. 
Then $B_{r_k}\subset\R^2\setminus\2{\Omega}_k$ for all $k\ge k_0$. 
Hence $\Omega_k$ is contractible to a point in 
$\2{B}_{c_0r_k}\setminus B_{r_k}$ for all $k\ge k_0$. By (1), 
\begin{equation}
L(\gamma_k)=\int_0^{2\pi}(g_{ij}\overset{.}\gamma_k^i\overset{.}\gamma_k^j)
^{\frac{1}{2}}\,d\phi\ge\sqrt{\lambda_1(c_0r_k)}L_e(\gamma_k)\quad
\forall k\ge k_0.
\end{equation}
By the isoperimetric inequality,
\begin{equation}
4\pi|\Omega_k|\le L_e(\gamma_k)^2.
\end{equation}
Then by (40) and (41),
\begin{equation}
L(\gamma_k)\ge 2(\pi\lambda_1(c_0r_k)|\Omega_k|)^{\frac{1}{2}}
\quad\forall k\ge k_0.
\end{equation}
Now
\begin{equation}
A_{in}(\gamma_k)=\int_{\Omega_k}\sqrt{\mbox{det}g_{ij}}\,dx
\le\lambda_2(r_k)|\Omega_k|\quad\forall k\ge k_0.
\end{equation}
By (5), (42) and (43),
\begin{align}
&L(\gamma_k)\ge 2\pi^{\frac{1}{2}}\left(\frac{\lambda_1(c_0r_k)}
{\lambda_2(r_k)}\right)^{\frac{1}{2}}A_{in}(\gamma_k)^{\frac{1}{2}}
\ge 2(\pi\delta)^{\frac{1}{2}}A_{in}(\gamma_k)^{\frac{1}{2}}
\quad\forall k\ge k_0\nonumber\\
\Rightarrow\quad&I(\gamma_k)\ge\frac{L(\gamma_k)}{A_{in}(\gamma_k)}
\ge2(\pi\delta)^{\frac{1}{2}}A_{in}(\gamma_k)^{-\frac{1}{2}}
\quad\forall k\ge k_0.
\end{align}
Since $\Omega_k\subset\R^2\setminus B_{r_k}$, 
\begin{equation}
A_{in}(\gamma_k)\to 0\quad\mbox{ as }k\to\infty.
\end{equation}
Letting $k\to\infty$ in (44) by (45) we get $I=\infty$. This contradicts 
\eqref{isoperimetric-ratio-condition}. Hence \eqref{0-out} does not hold 
and the lemma follows. 
\end{proof}

By Lemma 4 there exists a constant $a_1>r_0$ such that
\begin{equation}
r_k\le a_1\quad\forall k\in\Z^+.
\end{equation}

\begin{lem} 
$\gamma_k\in\2{B}_{a_1^2}\quad\forall k\in\Z^+$.
\end{lem}
\begin{proof} Let $\rho_k=\max_{\gamma_k}|x|$. Suppose the lemma does not hold. 
Then there exists a subsequence of $\rho_k$ which we may assume without loss 
of generality to be the sequence itself such that
\begin{equation}
\rho_k>a_1^2\quad\forall k\in\Z^+.
\end{equation}
By (1), (4), (46), (47) and an argument similar to the proof of (24),
\begin{equation}
L(\gamma_k)\ge\int_{a_1}^{a_1^2}\sqrt{\lambda_1(\rho)}\,d\rho\ge b_2
\quad\forall k\in\Z^+.
\end{equation}
Hence by (48),
\begin{align*}
&I(\gamma_k)=\frac{AL(\gamma_k)}{A_{in}(\gamma_k)A_{out}(\gamma_k)}
\ge\frac{Ab_2}{(A/2)^2}=\frac{4b_2}{A}\quad\forall k\in\Z^+\\
\Rightarrow\quad&I\ge\frac{4b_2}{A}\quad\mbox{ as }k\to\infty.
\end{align*}
This contradicts \eqref{isoperimetric-ratio-condition} and the definition of
$b_0$. Hence the lemma follows.
\end{proof} 

Let $L_k=L(\gamma_k)$. Since $\2{B}_{a_1^2}$ is compact, there exists 
constants $c_2>c_1>0$ such that
\begin{equation}
c_1\delta_{ij}\le g_{ij}\le c_2\delta_{ij}\quad\mbox{ on }
\2{B}_{a_1^2}.
\end{equation}

\begin{lem}\label{Lk-lower-bd} 
There exists a constant $\delta_1>0$ such that $L_k\ge\delta_1\quad
\forall k\in\Z^+$.
\end{lem}
\begin{proof} 
By (49),
\begin{equation}\label{g-euclidean-compare}
\left\{\begin{aligned}
&c_1^{\frac{1}{2}}L_e(\gamma_k)\le L_k\le c_2^{\frac{1}{2}}
L_e(\gamma_k)\quad\forall k\in\Z^+\\
&c_1|\Omega_k|\le A_{in}(\gamma_k)\le c_2|\Omega_k|\quad\forall k\in\Z^+.
\end{aligned}\right.
\end{equation}
By \eqref{sequence-limit}, (41) and \eqref{g-euclidean-compare},
\begin{align*}
&b_0>\frac{L_k}{A_{in}(\gamma_k)}\ge\frac{c_1^{\frac{1}{2}}L_e(\gamma_k)}
{c_2|\Omega_k|}\ge\frac{c_1^{\frac{1}{2}}}{c_2}\cdot\frac{L_e(\gamma_k)}
{(L_e(\gamma_k)^2/4\pi)}\ge\frac{4\pi c_1^{\frac{1}{2}}}{c_2L_e(\gamma_k)}
\quad\forall k\in\Z^+\\
\Rightarrow\quad&L_k\ge c_1^{\frac{1}{2}}L_e(\gamma_k)
\ge\frac{4\pi c_1}{c_2b_0}\quad\forall k\in\Z^+
\end{align*}
and the lemma follows.
\end{proof}

By the proof of Lemma \ref{Lk-lower-bd} we have the following corollary.

\begin{cor}
For any constant $C_1>0$ there exists a constant $\delta_1>0$ such that 
\begin{equation*}
L(\gamma)>\delta_1
\end{equation*}
for any simple closed curve $\gamma\subset\2{B}_{a_1^2}$ satisfying
\begin{equation}\label{isoperimetric-ratio-upper-bd}
I(\gamma)<C_1.
\end{equation}
\end{cor}

By \eqref{isoperimetric-ratio} and Corollary 7 we have the following corollary.

\begin{cor}
For any constant $C_1>0$ there exists a constant $\delta_2>0$ such that 
\begin{equation*}
A_{in}(\gamma)>\delta_2\quad\mbox{ and }\quad A_{out}(\gamma)>\delta_2
\end{equation*}
for any simple closed curve $\gamma\subset\2{B}_{a_1^2}$ satisfying 
\eqref{isoperimetric-ratio-upper-bd}.
\end{cor}

\begin{lem}
There exists a constant $C_2>0$ such that the following holds. Suppose
$\beta\subset\2{B}_{a_1^2}$ is a closed simple curve. Then 
under the curve shrinking flow
\begin{equation}\label{csf}
\frac{\1\beta}{\1\tau}(s,\tau)=k\vec{N}
\end{equation}
with $\beta (s,0)=\beta (s)$ where for each $\tau\ge 0$, $k(\cdot,\tau)$ is 
the curvature, $\vec{N}$ is the unit inner normal, and $s$ is the arc length 
of the curve $\beta(\cdot,\tau)$ with respect to the metric $g$, 
there exists $\tau_0\ge 0$ such that the curve $\beta^{\tau_0}
=\beta(\cdot,\tau_0)\subset \2{B}_{a_1^2}$ satisfies $I(\beta^{\tau_0})
\le I(\beta)$ and
$$
\int k(s,\tau_0)^2\,ds\le C_2.
$$
\end{lem}
\begin{proof}
Since the proof is similar to the proof of \cite{DH} and the Lemma on P.197 
of \cite{H2}, we will only sketch the proof here. Let $\beta^{\tau}
=\beta(\cdot,\tau)$ and write
$$
L(\tau)=L_g(\beta (\cdot,\tau)),\,\, I(\tau)=I(\beta^{\tau})
=I_g(\beta (\cdot,\tau)),
$$
and the areas
$$
A_{in}(\tau)=A_{in}(\beta (\cdot,\tau)),\,\, A_{out}(\tau)
=A_{out}(\beta (\cdot,\tau)),
$$
with respect to the metric $g$. Let $T_1>0$ be the maximal existence time
of the solution of \eqref{csf}. Then
\begin{equation}\label{beta-bd-set}
\beta^{\tau}\subset\2{B}_{a_1^2}\quad\forall 0\le\tau<T_1.
\end{equation}
Similar to the result on P.196 of \cite{H2} we have
\begin{equation}\label{csf-eqns}
\frac{\1 A_{in}}{\1\tau}=-\int k\,ds, \quad \frac{\1 A_{out}}{\1\tau}
=\int k\,ds,\quad \frac{\1 L}{\1\tau}=-\int k^2\,ds
\end{equation}
and
\begin{equation}\label{Gauss-Bonnet}
\int k\,ds+\int_{\Omega(\tau)} KdV_g=2\pi
\end{equation}
by the Gauss-Bonnet theorem where $K$ is the Gauss curvature with respect
to $g$ and $\Omega(\tau)\subset\2{B}_{a_1^2}$ is the region enclosed 
by the curve 
$\beta (s,\tau)$. Let $C_1=2I(\beta)$. By continuity there exists a constant
$0<\delta_0<T_1$ such that 
\begin{equation}\label{compact}
I(\tau)<C_1\quad\forall 0\le\tau\le\delta_0.
\end{equation}
By (56), Corollary 7, and Corollary 8 there exist constants $\delta_1>0$, 
$\delta_2>0$, such that
\begin{equation}\label{lower-bds}
L(\tau)>\delta_1,\quad A_{in}(\tau)>\delta_2,\quad A_{out}(\tau)>\delta_2
\quad\forall 0\le\tau\le\delta_0.
\end{equation}
Now
\begin{equation}\label{logL-eqn}
\frac{\1}{\1\tau}(\log I(\tau))=\frac{1}{L}\frac{\1 L}{\1\tau}
-\frac{1}{A_{in}}\frac{\1 A_{in}}{\1\tau}-\frac{1}{A_{out}}\frac{\1 A_{out}}
{\1\tau}+\frac{1}{A}\frac{\1 A}{\1\tau}.
\end{equation}
By \eqref{beta-bd-set} and \eqref{Gauss-Bonnet} $\int k\,ds$ is uniformly 
bounded for all $0\le\tau<T_1$.
Then by \eqref{csf-eqns}, \eqref{Gauss-Bonnet}, \eqref{lower-bds},
and \eqref{logL-eqn}, there
exists a constant $C_2>0$ independent of $\delta_0$ such that
\begin{equation*}
\frac{\1}{\1\tau}(\log I(\tau))<0 
\end{equation*}
for any $\tau\in (0,\delta_0]$ satisfying 
\begin{equation*}
\int k(s,\tau)^2\,ds>C_2.
\end{equation*}
If 
$$
\int k(s,0)^2\,ds\le C_2,
$$
we set $\tau_0=0$ and we are done. If 
$$
\int k(s,0)^2\,ds>C_2,
$$
then either there exists $\tau_0\in (0,\delta_0]$ such that 
\begin{equation}\label{csf-case1}
\int k(s,\tau_0)^2\,ds=C_2\quad\mbox{ and }\quad \int k(s,\tau)^2\,ds>C_2
\quad\forall 0\le\tau<\tau_0
\end{equation}
or 
\begin{equation}\label{csf-case2}
\int k(s,\tau)^2\,ds>C_2\quad\forall 0\le\tau\le\delta_0.
\end{equation}
If \eqref{csf-case1} holds, since $I(\tau_0)\le I(0)$ we are done. 
If \eqref{csf-case2} holds, since $I(\delta_0)\le I(0)$ we can
repeat the above the argument a finite number of times. Then either 

\begin{flushleft}
(a) there exists $\tau_0\in (0,T_1)$ such that \eqref{csf-case1} holds 
\end{flushleft}
or 
\begin{flushleft}
(b) 
\begin{equation}\label{csf-case3}
\int k(s,\tau)^2\,ds>C_2\quad\forall 0\le\tau<T_1
\end{equation} 
holds. 
\end{flushleft}

If (b) holds, then similar to the proof of the Lemma on P.197
of \cite{H2} by \eqref{lower-bds} we get a contradiction to the 
Grayson theorem (\cite{H2},\cite{Gr1},\cite{Gr2}) for curve shortening 
flow. Hence (a) holds. Since $I(\tau_0)\le I(0)$, the lemma follows.
\end{proof}

To complete the proof of Theorem 1 we also need the following technical 
lemma.

\begin{lem}
For any positive numbers $\alpha_1,\alpha_2,A_1,A_2,A_3$ we have
\begin{equation}\label{algebraic-ineqn}
(\alpha_1+\alpha_2)\left(\frac{1}{A_2}+\frac{1}{A_1+A_3}\right)
\ge\min\left\{\alpha_1\left(\frac{1}{A_1}+\frac{1}{A_2+A_3}\right),
\alpha_2\left(\frac{1}{A_3}+\frac{1}{A_1+A_2}\right)\right\}.
\end{equation}
\end{lem}
\begin{proof}
Suppose \eqref{algebraic-ineqn} does not hold. Then
\begin{align}\label{algebraic-ineqn1}
&(\alpha_1+\alpha_2)\left(\frac{1}{A_2}+\frac{1}{A_1+A_3}\right)
\le\alpha_1\left(\frac{1}{A_1}+\frac{1}{A_2+A_3}\right)\nonumber\\
\Rightarrow\quad&\frac{A_1(A_2+A_3)}{A_2(A_1+A_3)}
\le\frac{\alpha_1}{\alpha_1+\alpha_2}
\end{align}
and
\begin{align}\label{algebraic-ineqn2}
&(\alpha_1+\alpha_2)\left(\frac{1}{A_2}+\frac{1}{A_1+A_3}\right)
\le\alpha_2\left(\frac{1}{A_3}+\frac{1}{A_1+A_2}\right)\nonumber\\
\Rightarrow\quad&\frac{A_3(A_1+A_2)}{A_2(A_1+A_3)}
\le\frac{\alpha_2}{\alpha_1+\alpha_2}.
\end{align}
Summing \eqref{algebraic-ineqn1} and \eqref{algebraic-ineqn2},
$$
\frac{2A_1A_3}{A_2(A_1+A_3)}\le 0\quad\Rightarrow\quad A_1=0\mbox{ or }
A_3=0. 
$$
Contradiction arises. Hence  \eqref{algebraic-ineqn} holds and the lemma 
follows.
\end{proof}

We are now ready for the proof of Theorem 1.

\noindent{\it Proof of Theorem 1}:
Since the proof is similar to the proof of \cite{H1} and \cite{H2} we
will only sketch the argument here. Let $C_2>0$ be given by Lemma 9
and $\delta_1>0$ be given by Corollary 7 with $C_1=b_0$.
By Lemma 5, Lemma 6, Corollary 7, Lemma 9 and an argument similar to the 
proof of \cite{H2} for each $j\in\Z^+$ there exists a closed simple curve 
$\2{\gamma}_j\subset\2{B}_{a_1^2}$ satisfying
$$
I(\2{\gamma}_j)\le I(\gamma_j)\quad\mbox{and}\quad L(\2{\gamma}_j)
\ge\delta_1\quad\forall j\in\Z^+
$$
and
\begin{equation}\label{k-l2-uniform-bd}
\int_{\2{\gamma}_j}k^2\,ds\le C_2
\end{equation}
where $k$ is the curvature of $\2{\gamma}_j$. By \eqref{k-l2-uniform-bd}
and the same argument as that on P. 197-199 of \cite{H2} $\2{\gamma}_j$
are locally uniformly bounded in $L_2^1$ and $C^{1+\frac{1}{2}}$. Hence 
$\2{\gamma}_j$ has a sequence which we may assume without loss of 
generality to be the sequence itself that converges uniformly in 
$L_p^1$ for any $1<p<2$ and in $C^{1+\alpha}$ for any 
$0<\alpha<1/2$ as $j\to\infty$ to some closed immersed curve 
$\gamma\subset\2{B}_{a_1^2}$. Moreover $\gamma$ satisfies
$$
I=I(\gamma)\quad\mbox{ and }\quad L(\gamma)\ge\delta_1.
$$ 
Since $\gamma$ is the limit of embedded curves, $\gamma$ cannot cross
itself and at worst it will be self tangent. Suppose $\gamma$ is
self tangent. Without loss of generality we may assume that $\gamma$
is only self tangent at one point. Then $\gamma=\beta_1\cup\beta_2$
with $\beta_1\cap\beta_2$ being a single point 
where $\beta_1$, $\beta_2$, are simple closed curves. Then
$A_{in}(\gamma)=A_{in}(\beta_1)+A_{in}(\beta_2)$, $A_{out}(\beta_1)=
A_{out}(\gamma)+A_{in}(\beta_2)$, $A_{out}(\beta_2)=
A_{out}(\gamma)+A_{in}(\beta_1)$, and $L(\gamma)=L(\beta_1)+L(\beta_2)$. 
Let $L_1=L(\beta_1)$ and $L_2=L(\beta_2)$. By Lemma 10,
\begin{align*}
&(L_1+L_2)\left(\frac{1}{A_{out}(\gamma)}
+\frac{1}{A_{in}(\beta_1)+A_{in}(\beta_2)}\right)\\
\ge&\min\left\{L_1
\left(\frac{1}{A_{in}(\beta_1)}+\frac{1}{A_{out}(\gamma)+A_{in}(\beta_2)}\right),
L_2\left(\frac{1}{A_{in}(\beta_2)}+\frac{1}{A_{out}(\gamma)
+A_{in}(\beta_1)}\right)\right\}.
\end{align*}
Hence
\begin{align*}
L(\gamma)&\left(\frac{1}{A_{out}(\gamma)}+\frac{1}{A_{in}(\gamma)}\right)\\
\ge&\min\left\{L_1\left(\frac{1}{A_{in}(\beta_1)}+\frac{1}{A_{out}(\beta_1)}
\right),
L_2\left(\frac{1}{A_{in}(\beta_2)}+\frac{1}{A_{out}(\beta_2)}\right)\right\}\\
\Rightarrow\quad I(\gamma)\ge&\min (I(\beta_1),I(\beta_2))\\
\Rightarrow\quad I(\gamma)=&\min (I(\beta_1),I(\beta_2)).
\end{align*}
Without loss of generality we may assume that $I(\gamma)=I(\beta_1)$. Then 
$\beta_1$ is a simple closed curve which attains the minimum. Similar to
the proof of \cite{H2}, by a variation argument $\beta_1$ has constant 
curvature
$$
k=L\left(\frac{1}{A_{in}}-\frac{1}{A_{out}}\right).
$$
Hence $\beta_1$ is smooth and the theorem follows.

\hfill$\square$
\end{document}